\documentclass[A4,11pt]{article}

\usepackage{geometry}

\usepackage{authblk}

\usepackage{amssymb,amsmath,amsthm}

\usepackage{mathtools}
\usepackage[linktoc=all,hidelinks]{hyperref}

\DeclareMathOperator{\ord}{ord}

\DeclareMathOperator{\supp}{supp}
\DeclareMathOperator{\diag}{diag}

\newcommand{\bra}[1]{\langle#1\rangle}

\newtheorem{Thm}{Theorem}
\newtheorem*{Thm*}{Theorem}
\newtheorem{Lm}{Lemma}
\newtheorem{Cor}{Corollary}

\title{Spectrality of Prime Size Tiles}
\author{Weiqi Zhou\thanks{zwq@xzit.edu.cn}}
\affil{\small School of Mathematics and Statistics, Xuzhou University of Technology \\  {\footnotesize Lishui Road 2, Yunlong District, Xuzhou, Jiangsu Province, China 221018}}
\date{}							

\begin{document}
\maketitle
\begin{abstract}
We prove that if a tile in $\mathbb Z^d$ has prime size $p$, then it must be spectral. The proof is by contradiction, it is simply shown that the tiling complement of such a tile can not annihilate all $p$-subgroups. In addition, with a simple transformation we prove that any $p$ points in general linear positions in $\mathbb Z^d (d\ge p-1)$ must be both tiling and spectral. \\

{\noindent
}
\end{abstract}
 
\section{Introduction}
Let $G$ be a finitely generated Abelian group, a set $A$ is \emph{tiling} (or is a \emph{tile}) in $G$ if there exists some $B$ such that $\{A+b\}_{b\in B}$ forms a partition of $G$, and is \emph{spectral} if there exists a set $S$ of group characters of $G$ that forms an orthogonal basis on $L^2(A)$. In such cases $B$ (resp. $S$) is said to be a \emph{tiling complement} (resp. \emph{spectrum}) of $A$, while $(A,B)$ (resp. $(A,S)$) is called a tiling pair (resp. spectral pair) in $G$, and we often write $G=A\oplus B$. 

Fuglede conjectured that being spectral and being tiling are actually equivalent \cite{fuglede1974}, this has been studied extensively and disproved in general when the dimension is at least three \cite{farkas2006, matolcsi2006, kolounzakis2006, matolcsi2005, tao2004} but remains open for one dimensional and two dimensional cases. Researches over the Fuglede conjecture in finite Abelian groups as well as in $p$-adic fields have been active in the past years, see e.g., \cite{fallon2022, fan2016, iosevich2017, laba2022, laba2025, wang1997, malikiosis2022, zhang2021} (the full list is of course too long to be enumerated here, this is only an attempt to mention some of the results in one dimensional or two dimensional cases, see perhaps references therein for more on the history and development of this problem).

It is well known that a tiling set of prime size in $\mathbb Z^d$ must have at least one periodic tiling complement \cite{szegedy1998} (it may be worthing mentioning that being periodic in $d$ dimensions means having $d$ linearly independent periods, while having periods in one direction is called weakly periodic instead, readers may also find the periodic tiling conjecture and its positive and negative cases in e.g., \cite{bhattacharya2020, tao2021, tao2024, wang1996} relevant). The aim of this short note is to establish the following statements (Theorem \ref{ThmMain} and \ref{ThmMain2} in the last section):

\begin{Thm*}
Let $p$ be a prime number and $A\subset\mathbb Z^d$ with $|A|=p$, if $A$ tiles $\mathbb Z^d$, then it is also spectral.
\end{Thm*}

\begin{Thm*}
If $p$ is a prime number, then any $p$ points in general linear positions in $\mathbb Z^d (d\ge p-1)$ must be both tiling and spectral.
\end{Thm*}

Here $p$ points being in \emph{general linear positions} means that they are not contained in any hyperplane or affine hyperplane of dimension $p-2$, alternatively this means if we fix any one of these $p$ points, then the $p-1$ vectors formed from the fixed point to the other $p-1$ points are linearly independent in $\mathbb Z^d$, therefore this necessarily requires $d\ge p-1$. The second theorem leads to some interesting findings, for example, it follows that any $3$ points in $\mathbb Z^d (d>1)$ must be both spectral and tiling if they are not on the same line. It is also worth mentioning that $p$ linearly independent points must be tiling is already included in \cite[Corollary 2.4]{matolcsi2005}. 

The proof of the first theorem applies several techniques: First we reduce the problem to $\mathbb Z_n^d$ by using periodicity, then we produce a size $p$ spectrum from $p$-subgroups of $\mathbb Z_n^d$. The construction involves dissecting the group into equivalence classes formed by generators of cyclic subgroups, from each equivalence class we can squeeze out a derived set of size $p$ whose difference set is still in the same class. The existence of a spectrum within one of these classes is then shown by contradiction: Projecting onto $\mathbb Z_{p^k}^d$ (where $p^k$ is the largest power of $p$ that divides $n$) and invoking the uncertainty principle, we are able to show that if the tiling complement annihilates all $p$-subgroups (which will be the case if $A$ annihilates none), then its size must be divisible by $p^{kd}$ (a statement that is similar to the Coven-Meyerowitz T1 property in \cite{coven1999}), thus leading to a contradiction if $A$ does not admit a spectrum.

The proof of the second theorem then reduces to proving any $p$ points in general linear positions are tiling, for which we can simply show that they tile the sublattice they generate. This follows from the fact that the set formed by the origin and $(1,0\ldots,0)$, $(0,2,0,\ldots,0)$, $\ldots$, $(0,\ldots,0,p-1)$ is tiling (a tiling complement can be found explicitly by computation), then we can map this set to the $p$ points we are looking at by a linear transform.

\section{Preliminaries}
Let $f$ be a function on $\mathbb Z_n^d$, then its Fourier transform on $\mathbb Z_n^d$ is 
$$\hat f(\xi)=\sum_{x\in \mathbb Z_n^d}f(x)\cdot e^{2\pi i\bra{x,\xi}/n}, \quad \xi\in\mathbb Z_n^d.$$
Denote by 
$$\supp(f)=\{x\in\mathbb Z_n^d: f(x)\neq 0\}, \quad Z(f)=\{x\in\mathbb Z_n^d: f(x)=0\},$$
the support and the zero set of $f$ respectively, then the \emph{uncertainty principle} on finite Abelian groups \cite{donoho1989, smith1990} asserts that
\begin{equation} \label{EqUncertainty}
|\supp(f)|\cdot |\supp(\widehat{f})|\ge |\mathbb Z_n^d|,
\end{equation}
holds for any non-zero function $f$ defined on $\mathbb Z_n^d$.

If $A$ is a finite multiset (i.e., a set that allows repetitive elements) on $\mathbb Z_n^d$ (i.e., each element of $A$ is a member of $\mathbb Z_n^d$), and the multiplicity of an element $a\in A$ (i.e., the number of copies of $a$ in $A$) is denoted by $m_a$, then the characteristic function on such a multiset $A$ can be defined and written as
$$\mathbf 1_A(x)=\begin{cases}m_x & x\in A, \\ 0 & x\notin A. \end{cases}$$
Its Fourier transform on $\mathbb Z_n^d$ is 
$$\widehat{\mathbf 1}_A(\xi)=\sum_{a\in A}m_a\cdot e^{2\pi i\bra{a,\xi}/n}.$$

If $A,B$ are finite multisets, and $A+B$ is the multiset obtained by enumerating $a\in A, b\in B$ (so $a$ appears $m_a$ times and $b$ appears $m_b$ times) and take the collection of all sums $a+b$, then straightforward computation shows
$$\widehat{\mathbf 1}_{A+B}=\widehat{\mathbf 1}_A\cdot \widehat{\mathbf 1}_B.$$

With these settings we easily get that when $A,B,S$ are usual sets (every element appear only once), then
\begin{align}
(A,S) \text{ is a spectral pair on } \mathbb Z_n^d &\quad\Leftrightarrow\quad  \begin{cases}|A|=|S|, \\ \Delta S\subseteq Z(\widehat{\mathbf 1}_A), \end{cases} \\
(A,B) \text{ is a tiling pair on } \mathbb Z_n^d &\quad\Leftrightarrow\quad \begin{cases}|A|\cdot|B|=n^d, \\ \Delta \mathbb Z_n^d\subseteq Z(\widehat{\mathbf 1}_A)\cup Z(\widehat{\mathbf 1}_B),\end{cases} \label{EqTile} 
\end{align}
where the difference set $\Delta S$ is defined as
$$\Delta S=\{s-s': s,s'\in S,\; s\neq s'\}.$$

One may also verify the \emph{Poisson summation formula} below:
\begin{equation} \label{EqPoisson}
\widehat{\mathbf 1_H}=|H|\cdot\mathbf 1_{H^{\perp}},
\end{equation}
where $H\lhd \mathbb Z_n^d$ is a subgroup and 
$$H^{\perp}=\{x\in\mathbb Z_n^d: \bra{x,h}=0, \forall h\in H\},$$
is its \emph{orthogonal group} in $\mathbb Z_n^d$. \eqref{EqPoisson} can be first established on cyclic subgroups (which is quite trivial), then since non-cyclic subgroups are products of cyclic ones, its orthogonal set is simply the intersection over orthogonal sets of its cyclic factors.

\begin{Lm} \label{LmMinimal}
Let $A$ be a multiset on $\mathbb Z_n^d$, if $\mathbb Z_n^d\setminus\{0\}\subseteq Z(\widehat{\mathbf 1}_A)$, then $\mathbf 1_A$ must be a scalar multiple of $\mathbf 1_{\mathbb Z_n^d}$, i.e., $\mathbf 1_A=m\mathbf 1_{\mathbb Z_n^d}$ for some $m\in\mathbb N$.
\end{Lm}

\begin{proof}
Assume the contrary that $\mathbf 1_A$ is not a scalar multiple of $\mathbf 1_{\mathbb Z_n^d}$, then let $m$ be the smallest multiplicity of elements in $A$, i.e.,
$$m=\min\{m_a: a\in A\},$$
and consider
$$f=\mathbf 1_{A}-m\mathbf 1_{\mathbb Z_n^d}.$$
With this construction we would have $f(a)=0$ for all $a$ whose multiplicity in $A$ is $m$ and $f$ is not identically $0$ (since $\mathbf 1_A$ is assumed not to be a scalar multiple of $\mathbf 1_{\mathbb Z_n^d}$). Therefore we get
\begin{equation} \label{EqSuppF}
0<|\supp(f)|<|\mathbb Z_n^d|,
\end{equation}

On the other hand, it is straightforward to check that
$$Z(\widehat{\mathbf 1}_{\mathbb Z_n^d})=\mathbb Z_n^d\setminus\{0\},$$
thus combined with the assumption that $\mathbb Z_n^d\setminus\{0\}\subseteq Z(\widehat{\mathbf 1}_A)$ we get
$$\mathbb Z_n^d\setminus\{0\}\subseteq Z(\widehat f),$$
i.e.,
$$\supp(\widehat f)=\{0\}.$$
Together with \eqref{EqSuppF} this leads to
$$|\supp(f)|\cdot |\supp(\widehat{f})|=|\supp(f)|\cdot 1=|\supp(f)|<|\mathbb Z_n^d|,$$
which contradicts the uncertainty principle in \eqref{EqUncertainty}. 
\end{proof}

For any $x=(x_1,\ldots,x_d)\in\mathbb Z^d$, set
$$\pi_n(x)=(x_1\bmod n, \; \ldots, \; x_d\bmod n),$$
so that $\pi_n$ is the projection from $\mathbb Z^d$ to $\mathbb Z_n^d$ (as a map between sets). Given a finite set $A\subseteq\mathbb Z^d$, $\pi_n(A)$ will denote the multiset obtained by applying $\pi_n$ on every element of $A$. Clearly $\pi_n(A)$ will only be a usual set if (and only if) the map $A\mapsto \pi_n(A)$ is injective. 

\begin{Lm}[\cite{szegedy1998}] \label{LmPeriodicity}
Let $p$ be a prime number and $A\subset\mathbb Z^d$ with $|A|=p$, if $A\oplus B=\mathbb Z^d$, then there is some $n\in\mathbb N$ and $B'\in\mathbb Z_n^d$ such that $A\mapsto \pi_n(A)$ is injective and $\pi_n(A)\oplus B'=\mathbb Z_n^d$.
\end{Lm}
The way that Lemma \ref{LmPeriodicity} is stated here is slightly different from its original version in \cite[Theorem 17]{szegedy1998}, but it is clear that if $\{(t_1,0,\ldots,0),\ldots,(0,\ldots,0,t_d)\}$ is a set of periods of $B$, then the least common multiple of $t_1,\ldots,t_d$ can be taken as $n$, and $B'$ can be obtained by removing all duplications from $\pi_n(B)$.

\begin{Lm} \label{LmProj}
Let $n=p^km$ with $p$ being a prime number that does not divide $m$, and $A\subseteq \mathbb Z_n^d$. Set $A'=\pi_{p^k}(A)$, and let $f,g$ both be characteristic functions on $A'$, but defined on $\mathbb Z_n^d$ and $\mathbb Z_{p^k}^d$ respectively so that the Fourier transforms
$$\widehat{\mathbf 1}_A(x)=\sum_{a\in A}e^{2\pi i\bra{a,\xi}/n}, \quad \hat f(\xi)=\sum_{a'\in A'}m_{a'}e^{2\pi i\bra{a',\xi}/n}, \quad \xi\in\mathbb Z_n^d,$$
take place in $\mathbb Z_n^d$, while  
$$\hat g(\xi')=\sum_{a'\in A'}m_{a'}e^{2\pi i\bra{a',\xi'}/p^k}, \quad \xi'\in \mathbb Z_{p^k}^d,$$
is performed in $\mathbb Z_{p^k}^d$. If $x=(x_1,\ldots,x_d)\in\mathbb Z_n^d$ is an element of order $p^t$ for some $1\le t\le k$ (which means every coordinate $x_1,\ldots,x_d$ must be divisible by $m$), then we have
$$\widehat{\mathbf 1}_A(x)=\widehat{f}(x)=\widehat{g}(x'),$$
where $x'=(x_1/m,\ldots,x_d/m)\in\mathbb Z_{p^k}^d$.
\end{Lm}

\begin{proof}
Every $a=(a_1,\ldots,a_d)\in A$ can be written as
$$a=p^kq_a+r_a,$$
where $q_a$, $r_a$ respectively collect quotients and remainders of $a_1,\ldots, a_d$ divided by $p^k$. In particular we have
$$A'=\pi_{p^k}(A)=\{r_a: a\in A\}.$$

Then direct computation shows that 
\begin{align*}
\widehat{\mathbf 1_A}(x)=\sum_{a\in A}e^{2\pi i\bra{a,x}/n}=\sum_{a\in A}e^{2\pi i\bra{p^kq_a+r_a,mx'}/n}&=\sum_{a\in A}e^{2\pi i\bra{r_a,mx'}/n}, \\
&=\sum_{a\in A}e^{2\pi i\bra{r_a,x}/n} \left(=\widehat{f}(x)\right), \\
&=\sum_{a\in A}e^{2\pi i\bra{r_a,x'}/p^k} \left(=\widehat{g}(x')\right), \\
\end{align*}
which is the desired result.
\end{proof}

Given a finite Abelian group $G$, let $\sim$ be the equivalence relation on $G$ so that $g\sim g'$ if $g$ and $g'$ generate the same cyclic subgroup in $G$. If $E$ is an equivalence class under $\sim$, then we define the order of $E$, denoted by $\ord(E)$, to be the order of any element (as a member of the group $G$) in $E$ (by definition all elements in the same class will have the same order). 

\begin{Lm}\label{LmEquiv}
Let $E$ be an equivalence class under $\sim$ in $\mathbb Z_n^d$ and $A\subseteq \mathbb Z_n^d$. If $Z(\widehat{\mathbf 1}_A)\cap E\neq\emptyset$, then $E$ is completely contained in $Z(\widehat{\mathbf 1}_A)$, i.e., $E\subseteq Z(\widehat{\mathbf 1}_A)$.
\end{Lm}

\begin{proof}
Set $m=\ord(E)$ and $r=n/m$. Let $x$ be an arbitrary element of $E$, then $x=rx'$ for some $x'\in\mathbb Z_m^d$, consequently
$$\widehat{\mathbf 1}_A(x)=\sum_{a\in A}e^{2\pi i\bra{a,x}/n}=\sum_{a\in A}e^{2\pi i\bra{a,x'}/m}.$$
We shall view the above expression as the polynomial 
$$P(z)=\sum_{a\in A}z^{\bra{a,x'}},$$
evaluated at the $m$-th root of unity $\omega=e^{2\pi i/m}$, i.e.,
$$\widehat{\mathbf 1}_A(x)=P(\omega).$$
If $x\in Z(\widehat{1_A})$, then $P(\omega)=0$, which means $P(z)$ is divisible by the $m$-th cyclotomic polynomial $\Phi_m(z)$. On the other hand, if $y$ is another element of $E$, then by the definition of $E$ we must have $y=\lambda x$ for some $\lambda$ that is coprime to $m$, thus
$$\widehat{\mathbf 1}_A(y)=\sum_{a\in A}e^{2\pi i\bra{a,\lambda x}/n}=\sum_{a\in A}e^{2\pi i\lambda\bra{a,x}/n}=\sum_{a\in A}e^{2\pi i\lambda\bra{a,x'}/m}=P(\omega^\lambda)=0,$$  
where the last equality holds since $\lambda$ is coprime to $m$, and so that $\omega^\lambda$ is also a root of $\Phi_m(z)$. This finishes the proof since $x,y$ are chosen arbitrarily from $E$.
\end{proof}

A set $E'$ will be called a \emph{derived set} of $E$ if 
$$E'\subseteq E\cup\{0\}, \quad \Delta E'\subseteq E.$$

\begin{Lm}\label{LmDerived}
Let $E$ be an equivalence class under $\sim$ in some finite Abelian group. If $p$ is the smallest prime divisor of $\ord(E)$, then $E$ admits a derived set of size $p$. 
\end{Lm} 

\begin{proof}
Let $h$ be an arbitrary element of $E$, then we claim that
$$E'=\{0,h,2h,\ldots,(p-1)h\},$$
is a desired derived set. Indeed, that $|E'|=p$ is obvious, to see $\Delta E'\subseteq E$, set $n=\ord(E)$ and simply consider the group isomorphism from the cyclic group generated by $h$ to $\mathbb Z_n$ given by $\varphi: h\mapsto 1$. Let $F$ be the equivalence class formed by taking all generators of $\mathbb Z_n$, then $F$ consists of all numbers that are coprime to $n$ in $\mathbb Z_n$ and by this construction we have $\varphi(E)=F$. Now set $F'=\{0,1,\ldots,p-1\}$, then $\Delta F'=\{\pm 1,\ldots,\pm(p-1)\}$, clearly every member of $\Delta F'$ is coprime to $n$ since $p$ is the smallest divisor of $n$. Therefore $\Delta F'\subseteq F$, apply $\varphi^{-1}$ at both sides we get $\Delta E'\subseteq E$.
\end{proof}

In particular, Lemma \ref{LmDerived} holds for $\ord(E)$ being a power of $p$, which is how it will be applied later in this article.

\section{Main results}
\begin{Thm} \label{ThmMain}
Let $p$ be a prime number and $A\subset\mathbb Z^d$ with $|A|=p$, if $A$ tiles $\mathbb Z^d$, then it is also spectral.
\end{Thm}

\begin{proof}
Let $n$ and $B'$ be as asserted in Lemma \ref{LmPeriodicity}, and set $A'=\pi_n(A)$, so that we have $A'\oplus B'=\mathbb Z_n^d$. It suffices to show that $A'$ is spectral in $\mathbb Z_n^d$.

$A'$ being a tile in $\mathbb Z_n^d$ implies that $p$ divides $n^d$. Since $p$ is prime this also implies that $p$ divides $n$. Let $m$ be the number so that $n=p^km$ and $\gcd(p,m)=1$.

Consider all equivalence classes in $\mathbb Z_n^d$ whose orders are powers of $p$, by Lemma \ref{LmEquiv}, each equivalence class must be either completely inside $Z(\widehat{\mathbf 1}_{A'})$ or completely disjoint with $Z(\widehat{\mathbf 1}_{A'})$.

By Lemma \ref{LmDerived}, each of these equivalence classes has a derived set of size $p$, thus if any of these classes is in $Z(\widehat{\mathbf 1}_{A'})$, then the corresponding derived set would be a spectrum of $A'$ and we are done.

To see this must be the case, assume the contrary that all equivalence classes whose orders are powers of $p$ are disjoint with $Z(\widehat{\mathbf 1}_{A'})$, then by \eqref{EqTile} they must be annihilated by $\widehat{\mathbf 1}_{B'}$. By Lemma \ref{LmProj} this further implies 
\begin{equation} \label{EqPK}
\Delta \mathbb Z_{p^k}^d\subseteq Z(\widehat{\mathbf 1}_{\pi_{p^k}(B')}).
\end{equation}
Indeed, every element $x'\in \Delta\mathbb Z_{p^k}^d$ can be lifted to an element of the same order in $\mathbb Z_n^d$ through the map $x'\mapsto mx'$, and Lemma \ref{LmProj} asserts that 
\begin{equation} \label{EqB}
\widehat{\mathbf 1}_{\pi_{p^k}(B')}(x')=\widehat{\mathbf 1}_{B'}(mx').
\end{equation}
The order of $x'$ in $\mathbb Z_{p^k}^d$ is obviously a power of $p$, thus the order of $mx'$ in $\mathbb Z_n^d$ is also a power of $p$ (since it equals the order of $x'$ in $\mathbb Z_{p^k}^d$), which means by assumption the right-hand side of \eqref{EqB} must vanish. Repeating the argument on every member of $\Delta\mathbb Z_{p^k}^d$ leads to \eqref{EqPK}.

Now as $\pi_{p^k}(B')$ is a multiset on $\mathbb Z_{p^k}^d$, Lemma \ref{LmMinimal} indicates that $\mathbf 1_{\pi_{p^k}(B')}$ is just a scalar multiple of $\mathbf 1_{\mathbb Z_{p^k}^d}$, which means $|Z_{p^k}^d|$ divides $|\pi_{p^k}(B')|$. Since $|\pi_{p^k}(B')|=|B'|$, this further indicates that $p^{kd}$ divides $|B'|$, consequently we would have $p^{kd}m^d=|\mathbb Z_n^d|=|A|\cdot |B|$ is divisible by $p^{kd+1}$, which is a contradiction.
\end{proof}

\begin{Lm} \label{LmSimplex}
Let $e_i$ be the $i$-th standard Euclidean basis, then the set
$$A=\{0,e_1,2e_2,\ldots,(p-1)e_{p-1}\},$$
is both tiling and spectral in $\mathbb Z_p^{p-1}$.
\end{Lm}

\begin{proof}
Let $x=(1,\ldots,1)$, and set $\omega=e^{2\pi i/p}$, then it is easy to verify that
$$\widehat{\mathbf 1}_A(x)=1+\omega+\ldots+\omega^{p-1}=0.$$
Therefore the cyclic group generated by $x$ (denoted by $\bra{x}$) is a spectrum of $A$. To see that $A$ is tiling, we will construct a set $B$ so that $|A|\cdot|B|=|\mathbb Z_p^{p-1}|$ and 
\begin{equation} \label{EqZB}
Z(\widehat{\mathbf 1}_B)=\mathbb Z_p^{p-1}\setminus\bra{x},
\end{equation}
then it will follow from \eqref{EqTile} that $(A,B)$ forms a tiling pair.

Since $p$ is a prime number, each of $1,\ldots,p-1$ is $a$ $(p-1)$-st root of unity in the finite field $\mathbb F_p$, hence all entries of the $(p-1)\times (p-1)$ Fourier matrix (i.e., the Vandermonde matrix generated by these roots of unity $1,\ldots,p-1$) are in $\mathbb Z_p$, and the set of its columns, denoted by $u_1,\ldots,u_{p-1}$, is a basis (the Fourier basis) on $\mathbb Z_p^{p-1}$ (understood as a vector space over $\mathbb F_p$). In particular, since $x$ is a column of it we may without loss of generality order the columns in the way so that $u_1=x$.

Let $B$ be the span (over $\mathbb F_p$) of $u_2,\ldots,u_{p-1}$, then $|B|=p^{p-2}$ is of the correct size, and we simply notice that $B$ and $\bra{x}$ are orthogonal groups of each other, thus \eqref{EqZB} follows immediately from the Poisson summation formula \eqref{EqPoisson}.
\end{proof}

\begin{Cor} \label{CorBasis}
The set 
$$A=\{0,e_1,\ldots,e_{p-1}\},$$
is both tiling and spectral in $\mathbb Z^{p-1}$.
\end{Cor}

\begin{proof}
It suffices to show that $A$ is tiling in $\mathbb Z_p^{p-1}$ since if $(A,B)$ is a tiling pair in $\mathbb Z_p^{p-1}$, then $(A,B\oplus (p\mathbb Z)^{p-1})$ becomes a tiling pair in $\mathbb Z^{p-1}$.

Let 
$$S=\{0,e_1,2e_2,\ldots,(p-1)e_{p-1}\},$$
be the set defined in Lemma \ref{LmSimplex}, and set $D=\diag(1,2,\ldots,{p-1})$. Then $D$ is a group automorphism on $\mathbb Z_p^{p-1}$, its inverse is simply 
$D^{-1}=\diag(1,2^{-1},\ldots,(p-1)^{-1})$ where for each $k\in\{1,\ldots,p-1\}$, $k^{-1}$ is the multiplicative inverse of $k$ in $\mathbb Z_p$, each $k^{-1}$ exists since $p$ is a prime number.

By Lemma \ref{LmSimplex}, $S$ is tiling in $\mathbb Z_p^{p-1}$, therefore $A=D^{-1}S$ is also tiling in $D^{-1}\mathbb Z_p^{p-1}$, but $D^{-1}\mathbb Z_p^{p-1}$ equals $\mathbb Z_p^{p-1}$, thus $A$ is indeed tiling in $\mathbb Z_p^{p-1}$. Its spectrality then follows from Theorem \ref{ThmMain}.
\end{proof}

\begin{Thm} \label{ThmMain2}
If $p$ is a prime number, then any $p$ points in general linear positions in $\mathbb Z^d (d\ge p-1)$ must be both tiling and spectral.
\end{Thm}

\begin{proof}
In light of Theorem \ref{ThmMain}, it suffices to show that these points are tiling in $\mathbb Z^d$.

Without loss of generality we may assume that one of these points is the origin, and we denote the other $p-1$ points by $v_1,\ldots,v_{p-1}$. Let $V$ be an arbitrary linear transform that maps $e_i$ to $v_i$ for $i<=p-1$, such $V$ is a group isomorphism between $\mathbb Z^{p-1}$ and $V\mathbb Z^{p-1}$ since $0, v_1,\ldots,v_{p-1}$ are in general linear positions (which means that $v_1,\ldots,v_{p-1}$ are linearly independent). Further set
$$A=\{0,v_1,\ldots,v_{p-1}\}, \quad E=\{0,e_1,\ldots,e_{p-1}\}.$$

By Corollary \ref{CorBasis}, $E$ is tiling in $\mathbb Z^{p-1}$, thus $A=VE$ is also tiling in the sublattice $V\mathbb Z^{p-1}$. This implies that $A$ is also tiling in $\mathbb Z^d$. Its tiling complement can be produced explicitly: If $E'$ is a tiling complement of $E$ in $\mathbb Z^{p-1}$, and $F$ is the set formed by taking one representative from each coset of $V\mathbb Z^{p-1}$ (viewed as an additive subgroup) in $\mathbb Z^d$, then $B=(VE')\oplus F$ is a tiling complement of $A$.
\end{proof}

Theorem \ref{ThmMain2} produces some interesting consequences, for example we can assert that any $3$ points must be tiling as long as they are not on a line. On the other hand, this also shows that the condition that $p$ points have to be in general linear positions is tight, since $3$ colinear points indeed need not be tiling, e.g., $\{0,3,4\}$ does not tile $\mathbb Z$ (see also \cite{newman1977} for a characterization of tiles with prime power sizes in $\mathbb Z$). 

\section*{Acknowldgement}
The author would like to thank Tao Zhang (XDU) for the hospitality in Xi'an, and the Chern Institute of Mathematics (Tianjin) for hosting while preparing for this paper. The author thanks Jia Zheng (WHU) for mentioning the phenomenon described in Theorem \ref{ThmMain2} at $p=3$ to him, as well as Shilei Fan (CCNU) and Gergely Kiss (ARIM) for short communications that improved Lemma \ref{LmDerived}.

\end{document}